\documentclass{daj}
\usepackage{amsmath, amsthm, amsfonts, amssymb}
\usepackage{enumitem}
\usepackage{tabularx}
\usepackage{mathtools}
\numberwithin{equation}{section}

\theoremstyle{plain}
\newtheorem{theorem}[subsection]{Theorem}
\newtheorem{proposition}[subsection]{Proposition}
\newtheorem*{prop-no-number}{Proposition}
\newtheorem*{prop-}{Proposition}
\newtheorem{lemma}[subsection]{Lemma}

\theoremstyle{definition}
\newtheorem{defn}[subsection]{Definition}

\newtheorem{remark}[subsection]{Remark}

\renewcommand{\leq}{\leqslant}
\renewcommand{\geq}{\geqslant}
\renewcommand{\subset}{\subseteq}
\renewcommand{\supset}{\supseteq}

\def\E{\mathbb{E}}
\def\Z{\mathbb{Z}}

\def\C{\mathbb{C}}
\def\N{\mathbb{N}}
\def\P{\mathbb{P}}
\def\F{\mathbb{F}}

\def\Ghat{\widehat{G}}

\DeclareMathOperator{\Bohr}{Bohr}
\DeclareMathOperator{\rk}{rk}
\DeclareMathOperator{\Spec}{Spec}
\DeclareMathOperator{\Bin}{Bin}

\newcommand{\stirlingtwo}[2]{S_2(#1, #2)}

\providecommand{\abs}[1]{\lvert#1\rvert}
\providecommand{\norm}[1]{\lVert #1 \rVert}

\providecommand{\ceiling}[1]{\lceil#1\rceil}
\providecommand{\floor}[1]{\left\lfloor#1\right\rfloor}

\providecommand{\tup}[1]{{\vec{#1}}}
\newcommand{\brac}[1]{\left( #1\right)}

\dajAUTHORdetails{%
  title = {Logarithmic bounds for Roth's theorem via almost-periodicity}, 
  author = {Thomas F. Bloom and Olof Sisask},
  plaintextauthor = {Thomas F. Bloom, Olof Sisask}
  }
  
\dajEDITORdetails{%
   year={2019},
   number={4},
   received={30 October 2018},   
   published={10 May 2019},  
   doi={10.19086/da.7884},       
}   

\begin{document}

\begin{frontmatter}[classification=text]

\title{Logarithmic bounds for Roth's theorem via almost-periodicity} 

\author[tfb]{Thomas F. Bloom}
\author[os]{Olof Sisask}

\begin{abstract}
We give a new proof of logarithmic bounds for Roth's theorem on arithmetic progressions, namely that if $A \subset \{1,2,\ldots,N\}$ is free of three-term progressions, then $\abs{A} \leq N/(\log N)^{1-o(1)}$. Unlike previous proofs, this is almost entirely done in physical space using almost-periodicity. 
\end{abstract}
\end{frontmatter}

\section{Introduction}

We shall prove here the following version of Roth's theorem on arithmetic progressions.\footnote{For details of the asymptotic notation we use, see the next section.}

\begin{theorem}\label{thm:roth}
Let $r_3(N)$ denote the largest size of a subset of $\{1,2,\ldots,N\}$ with no non-trivial three-term arithmetic progressions. Then
\[ r_3(N) \ll \frac{N}{(\log N)^{1-o(1)}}. \]
\end{theorem}

Roth \cite{Ro:1953} proved this with a denominator of $\log \log N$ in the 1950s, laying the foundation for using harmonic analysis to tackle problems of an additive nature in rather arbitrary sets of integers. Subsequent improvements were made by Heath-Brown \cite{He:1987} and Szemer\'edi \cite{Sz:1990}, increasing the denominator to $(\log N)^c$ for some positive constant $c$, and then by Bourgain \cite{Bo:1999, Bo:2008}, obtaining such a bound with $c = \frac{1}{2}-o(1)$ and then $c = \frac{2}{3} - o(1)$. Sanders \cite{Sa:2012, Sa:2011} then proved this with $c = \frac{3}{4} - o(1)$ and was then the first to reach the logarithmic barrier in the problem, obtaining $c = 1-o(1)$. The best bounds currently known were then given by the first author \cite{Bl:2016},
\[ r_3(N) \ll \frac{(\log\log N)^4}{\log N} N. \]
Sanders's result \cite{Sa:2011} had a power of $6$ in place of the $4$ here, but the two techniques were quite orthogonal: \cite{Bl:2016} proceeds by getting structural information about the spectrum of the indicator function of a set $A$ with few three-term progressions, whereas \cite{Sa:2011} employed a result on the almost-periodicity of convolutions \cite{CrSi:2010} due to Croot and the second author, coupling this with a somewhat intricate combinatorial thickening argument on the physical side.

This article presents a fairly simple proof of logarithmic bounds for Roth's theorem, showing that they follow quite directly from almost-periodicity results along the lines of \cite{CrSi:2010}. Our focus is on clarity of exposition, and we therefore do not take steps to optimise the power of the $\log\log N$ term that we would obtain.

Some of the ideas in the present paper have been inspired by the authors' ongoing work on super-logarithmic bounds for Roth's theorem. In particular, there is a close relationship between $L^p$ norms of convolutions considered in this paper and the higher additive energies of the set of large Fourier coefficients used in the work of Bateman and Katz \cite{BaKa:2012} achieving super-logarithmic bounds in Roth's theorem over $\mathbb{F}_3^n$. 

\section{Notation, main theorem, and outline of proof}

\subsection*{Notation for averaging and counting}
The argument proceeds by studying high $L^p$-norms of the convolution $1_A*1_A$ of the indicator function of a set $A$ with itself. We use the following conventions for these objects. Let $G$ be a finite abelian group and let $f, g : G \to \C$ be functions. We define the convolution $f*g : G \to \C$ by
\[ f*g(x) = \sum_y f(y) g(x-y). \]
In considering $L^p$-norms on subsets of $G$, it will be convenient to sometimes use sums and to sometimes use averages. To distinguish between these, we write, for $B \subset G$, 
\[ \norm{f}_{\ell^p(B)}^p = \sum_{x \in B} \abs{f(x)}^p \qquad \text{and} \qquad \norm{f}_{L^p(B)}^p = \E_{x \in B} \abs{f(x)}^p, \]
where $\E_{x \in B} = \tfrac{1}{\abs{B}} \sum_{x \in B}$. If we write just $\norm{f}_p$ then we mean $\norm{f}_{L^p(G)}$. As usual $\norm{f}_\infty = \sup_{x \in G} \abs{f(x)}.$ We also write
\[ \langle f, g \rangle = \sum_{x \in G} f(x) \overline{g(x)}. \]
Finally, if $A \subset B \subset G$, we write $1_B$ for the indicator function of $B$, and $\mu_B$ for both the function $1_B/\abs{B}$ and for the measure $\mu_B(A) = \abs{A}/\abs{B}$; this latter quantity is known as the \emph{relative density} of $A$ in $B$. In the case $B = G$, this is known simply as the \emph{density} of $A$.

Where we have chosen discrete normalisations, the reader who is used to `compact normalisations' should find comfort in the fact that much of what we shall consider is normalisation-independent. For example, regardless of normalisation-convention, the function $1_A*\mu_B$ is always
\[ 1_A*\mu_B(x) = \E_{t \in B} 1_A(x-t). \]

We shall count three-term arithmetic progressions (3APs) across various sets. For $A, B, C \subset G$, with $2\cdot B \coloneqq \{ 2x : x \in B \}$, we write
\[ T(A,B,C) = \sum_{\substack{x,y,z \\ x+z=2y}} 1_A(x)1_B(y)1_C(z) = \langle 1_A*1_C, 1_{2\cdot B} \rangle \]
for the number of 3APs in $G$ with starting point in $A$, mid-point in $B$ and end-point in $C$. If $A=B=C$ we write just $T(A)$. Note that this counts also \emph{trivial} 3APs, where $x=y=z$.

\subsection*{Main theorem}
Our main theorem, then, is the following.

\begin{theorem}[Roth's theorem, counting version]\label{thm:rothCounting}
Let $G$ be a finite abelian group of odd order, and let $A \subset G$ be a set of density $\alpha > 0$. Then 
\[ T(A) \geq \exp\left(-C\alpha^{-1} (\log{2/\alpha})^C \right) \abs{A}^2 \]
where $C>0$ is an absolute constant. In particular, if $\alpha \geq (\log\log \abs{G})^C/\log \abs{G}$ then $A$ contains a non-trivial three-term arithmetic progression.
\end{theorem}

This immediately implies Theorem \ref{thm:roth}, by embedding a subset of $\{1,\ldots,N\}$ into \mbox{$G = \Z/(2N+1)\Z$} in the natural way, so that a (non-trivial) 3AP found in the set in $G$ is also a (non-trivial) 3AP in the original set.

To prove Theorem \ref{thm:rothCounting}, we employ a density increment strategy following the framework of Roth \cite{Ro:1953}.

\subsection*{Density increments}
Starting with $A \subset G$ of density $\alpha$, we show that if $A$ has few 3APs then there is a structured part $B \subset G$ --- in some cases a genuine subgroup --- such that some translate of $A$ has increased density on $B$:
\[ \mu_B\left( (A-x)\cap B\right) \geq (1+c)\alpha \]
where $c > 0$. Such a condition is succinctly summarised by $\norm{1_A*\mu_B}_\infty \geq (1+c)\alpha$. We then repeat the argument with $G$ replaced by $B$ and $A$ replaced by $A_2 \coloneqq (A-x)\cap B$: if $A_2$ has few 3APs, then we find a new structured piece and a new, denser subset, and repeat the argument. This cannot go on for too long, since the densities can never increase beyond $1$. At this point we will have shown that some translate of $A$ has many 3APs, which by translation-invariance of 3APs implies that $A$ itself does.

\subsection*{Outline of argument}
Finding the structured piece $B$ and the appropriate translate of $A$ relies on an almost-periodicity result for convolutions that says that $1_A*1_A$ is approximately translation-invariant in $L^p$ by something like a large subgroup. How we apply this depends on which of two cases we are in. 
If $\norm{1_A*1_A}_p$ is small, where $p \approx \log(1/\alpha)$, then the $L^{2p}$-almost-periodicity result is particularly efficient, and has as a straightforward consequence that if $T(A)$ deviates much from $\alpha \abs{A}^2$ then it must have a density increment on some subgroup-like object $B$. If, on the other hand, $\norm{1_A*1_A}_p$ is large, then, by $L^p$-almost-periodicity, we see that $\norm{1_A*1_A*\mu_B}_p$ must also be large for some group-like $B$, from which a density increment is immediate.

\subsection*{Asymptotic notation}
We employ both Vinogradov notation $X \ll Y$ and the `constantly changing constant'. Thus, any statement involving one or more expressions of the form $X_i \ll Y_i$ should be considered to mean ``There exist absolute constants $C_i > 0$ such that a true statement is obtained when $X_i \ll Y_i$ is replaced by $X_i \leq C_i Y_i$.'' Similarly, any sequence of statements involving unspecified constants $c, C$ should be read with the understanding that there exist positive constants to make the statements true, and that these constants may change from instance to instance. Generally the expectation will be that $c \leq 1$ and $C \geq 1$, a device intended to guide the reader.

\section{The finite field argument}

As is customary, we begin with a proof in the finite field case, as there are very few technical hurdles here. Our goal is the following density increment result.

\begin{theorem}\label{thm:FF}
If $A\subset \F_q^n$ has density $\alpha$ and $T(A)\leq \tfrac{\alpha}{2} \lvert A\rvert^2$ then there is a subspace $V$ with codimension $\lesssim_\alpha \alpha^{-1}$ such that $\norm{1_A\ast\mu_V}_\infty \geq \tfrac{5}{4}\alpha$.
\end{theorem}

The notation $X \lesssim_\alpha Y$ here means that $X \ll (\log(2/\alpha))^C\, Y$.

We prove this result by considering two possibilities: $\norm{\mu_A\ast 1_A}_{2m}$ is small for some large $m$, and $\norm{\mu_A\ast 1_A}_{2m}$ is large for some large $m$. It clearly suffices to show that both possibilities (combined with $T(A)\leq \alpha^3/2$) lead to a suitable density increment.

We will require the following almost-periodicity result. While it is not explicitly given in the literature, the deduction from the almost-periodicity results proved by Croot and the second author \cite{CrSi:2010} is routine, and is given in an appendix.

\begin{theorem}\label{thm:Lp}
Let $p \geq 2$ and $\epsilon \in (0,1)$. Let $G = \F_q^n$ be a vector space over a finite field and suppose $A\subset G$ has $\abs{A}\geq \alpha \abs{G}$. Then there is a subspace $V \leq G$ of codimension
\[ d \ll p \epsilon^{-2}\log(2/\epsilon)^2 \log(2/\alpha) \]
such that, for each $t \in V$,
\[ \norm{ \mu_A*1_A\ast \mu_V - \mu_A*1_A }_{p} \leq 
\epsilon\norm{\mu_A*1_A}_{p/2}^{1/2} + \epsilon^{2}. \]
\end{theorem}

\begin{lemma}\label{lem-modellow}
Suppose $A \subset \F_q^n$ has density $\alpha$ and $T(A) \leq \tfrac{\alpha}{2}\lvert A\rvert^2$. If $m \gg \log(2/\alpha)$ is such that
\[ \norm{ \mu_A*1_A }_{2m} \leq 10 \alpha, \]
then there is a subspace $V$ with codimension $\lesssim_\alpha m \alpha^{-1}$ such that $\norm{ 1_A*\mu_V }_\infty \geq \tfrac{5}{4}\alpha $.
\end{lemma}
\begin{proof}
	Apply Theorem \ref{thm:Lp} with $p = 4m$ and $\epsilon = \alpha^{1/2}/100$ to get a subspace $V$ of the required codimension such that
\begin{align*}
 \norm{\mu_A*1_A*\mu_V - \mu_A*1_A }_{4m}
 & \leq \epsilon \norm{\mu_A*1_A}_{2m}^{1/2} + \epsilon^{2}\\
 &\leq \frac{\alpha}{100}\brac{ \alpha^{-1/2}\norm{\mu_A\ast 1_A}_{2m}^{1/2}+1}\\
 &\leq \alpha/8
 \end{align*}
by our assumption on $\norm{\mu_A*1_A}_{2m}$. Now, if $1/r + 1/4m = 1$, H\"older's inequality gives
\begin{align*}
\norm{ \mu_A*1_A*1_{-2\cdot A}*\mu_V - \mu_A*1_A*1_{-2\cdot A} }_\infty &\leq \norm{ 1_{-2\cdot A} }_r \norm{\mu_A*1_A*\mu_V - \mu_A*1_A }_{4m} \\
    &= \alpha^{2-1/4m}/8\\
    &\leq \alpha^2/4.
\end{align*}
Since $\mu_A*1_A*1_{-2\cdot A}(0) \leq \alpha^2/2$ by assumption, this means that 
\[ 1_A*1_A*1_{-2\cdot A}*\mu_V(0)=\langle 1_A\ast 1_A\ast \mu_V, 1_{2\cdot A}\rangle \leq \tfrac{3}{4}\alpha^3. \]
It remains to convert this upper bound on the average into a lower bound for $\norm{1_A\ast \mu_V}_\infty$. There are a number of ways to do this, either in Fourier space or physical space; here we present a particularly short method using purely physical arguments. 

Suppose that $\norm{1_A\ast \mu_V}_\infty\leq (1+c)\alpha$, and let $f=(1+c)^{-1}\alpha^{-1}1_A\ast \mu_V$, so that $0\leq f\leq 1$. In particular, 
\[0\leq (1-f)\ast (1-f)=f\ast f-2\norm{f}_1+1=(1+c)^{-2}\alpha^{-2}1_A\ast 1_A\ast\mu_V-\frac{1-c}{1+c}.\]
It follows that
\[(1-c^2)\alpha^2 \leq 1_A\ast 1_A\ast \mu_V(x)\]
for all $x$. In particular, taking the inner product with $1_{2\cdot A}$ implies
\[(1-c^2)\alpha^3\leq \langle 1_A\ast 1_A\ast \mu_V, 1_{2\cdot A}\rangle\leq \frac{3}{4}\alpha^3,\]
and choosing $c=1/4$, say, gives a contradiction.
\end{proof}

On the other hand, if $\norm{\mu_A\ast 1_A}_{2m}$ is very large, then this directly implies a large density increment, without any assumptions on $T(A)$.

\begin{lemma}
If $\norm{\mu_A*1_A }_{2m} \geq 10\alpha$,
then there is a subspace $V$ of codimension $\lesssim_\alpha m \alpha^{-1}$ such that $\norm{ 1_A*\mu_V }_\infty \geq 5\alpha$.
\end{lemma}
\begin{proof}
Applying Theorem \ref{thm:Lp} as in the proof of Lemma~\ref{lem-modellow}, but with $p=2m$, there is a subspace $V$ of the required codimension such that
	\[ \norm{\mu_A*1_A*\mu_V - \mu_A*1_A }_{2m} \leq \frac{\alpha}{100}\brac{\alpha^{-1/2}\norm{\mu_A*1_A}_m^{1/2} +1}. \]
	It follows that
\begin{align*}
\norm{\mu_A*1_A*\mu_V}_{2m}
&\geq \norm{\mu_A*1_A}_{2m} - \frac{\alpha}{100}\brac{\alpha^{-1/2}\norm{\mu_A*1_A}_m^{1/2} +1} \\
&\geq \norm{\mu_A*1_A}_{2m} - \frac{\alpha}{100}\brac{\alpha^{-1/2}\norm{\mu_A*\mu_A}_{2m}^{1/2} +1}
\end{align*}
by nesting. Since $\norm{\mu_A*1_A}_{2m} \geq 10\alpha$, this is at least $5\alpha$, say. Hence
\[\norm{1_A\ast \mu_V}_\infty\geq \norm{\mu_A*1_A*\mu_V}_\infty \geq \norm{\mu_A\ast 1_A\ast \mu_V}_{2m}\geq 5\alpha , \]
and we have a density increment.
\end{proof}

The two preceding lemmas together immediately imply Theorem \ref{thm:FF}. A routine iterative application of this theorem then proves the finite field version of Theorem \ref{thm:rothCounting}: we can increase the density as in the theorem at most $C\log(1/\alpha)$ times before reaching $1$, and so a translate of $A$ must have plenty of 3APs on some subspace of codimension $\lesssim_\alpha \alpha^{-1}$.

\section{Bohr sets and $L^p$-almost-periodicity}
Following Bourgain \cite{Bo:1999}, the role played by subspaces in the density increment argument above will in general groups be played by Bohr sets, whose basic theory we review below. For proofs of these results, one may consult \cite{TaVu:2006}. Throughout, $G$ will be a finite abelian group, and we write $\Ghat = \{ \gamma : G \to \C^\times : \text{$\gamma$ a homomorphism}\}$ for the \emph{dual group} of $G$, the group operation being pointwise multiplication of functions. 
\begin{defn}[Bohr sets]
For a subset $\Gamma \subset \Ghat$ and a constant $\rho \geq 0$, we write 
\[ \Bohr(\Gamma, \rho) = \{ x \in G : \abs{\gamma(x)-1} \leq \rho \text{ for all $\gamma \in \Gamma$} \} \]
and call this a \emph{Bohr set}. Denoting it by $B$, we call $\rk(B) \coloneqq \abs{\Gamma}$ the \emph{rank} of $B$ and $\rho$ its radius.\footnote{$\Gamma, \rho$ cannot necessarily be read off from the set itself, but are considered part of the defining data.} 
We shall often need to narrow the radius: if $\tau \geq 0$, we write $B_\tau = \Bohr(\Gamma, \tau \rho)$. If furthermore $B' = \Bohr(\Lambda, \delta)$ where $\Lambda \supset \Gamma$ and $\delta \leq \rho$, then we write $B' \leq B$ and say that $B'$ is a \emph{sub-Bohr set} of $B$; note that this implies that $B' \subset B$ as sets.
\end{defn}

\begin{lemma}[Size estimates]\label{lemma:bohrDilateGrowth}
If $B$ is a Bohr set of rank $d$ and radius $\rho \leq 2$, then
\begin{enumerate}[topsep=0pt]
\item $\abs{B} \geq (\rho/2\pi)^d \abs{G}$,
\item $\abs{B_\tau} \geq (\tau/2)^{3d} \abs{B}$ for $\tau \in [0,1]$.
\end{enumerate}
\end{lemma}

One deficit of Bohr sets compared to subspaces is that the number of 3APs in a Bohr set $B$ need not be approximately $\abs{B}^2$ --- the trivial upper bound --- as it would be for a subspace. The standard work-around for this is to work with pairs $(B,B')$ of Bohr sets where $B'$ is a radius-narrowed copy of $B$. Provided $B$ is regular, defined as follows, one then has $T(B,B',B) \approx \abs{B}\abs{B'}$, matching the trivial upper bound.

\begin{defn}[Regularity]
We say that a Bohr set $B$ of rank $d$ is regular if 
\[ 1 - 12 d \abs{\tau} \leq \frac{\abs{B_{1+\tau}}}{\abs{B}} \leq 1 + 12 d \abs{\tau} \]
whenever $\abs{\tau} \leq 1/12d$.
\end{defn}

Note in particular that if $B$ is regular, then $\abs{B+B_{c/\rk(B)}} \leq 2\abs{B}$, for example. Importantly, regular Bohr sets are in plentiful supply, a fact that we use frequently:

\begin{lemma}\label{lemma:bohrReg}
If $B$ is a Bohr set, then there is a $\tau \in [\frac{1}{2}, 1]$ for which $B_\tau$ is regular.
\end{lemma}

Let us now assume that $G$ has odd order, so that the map $x \mapsto 2x$ is injective on $G$. The square-root map is then well-defined on $\Ghat$, and we write $\gamma^{1/2}$ for the unique element in $\Ghat$ such that $(\gamma^{1/2})^2 = \gamma$. We extend this to sets via $\Gamma^{1/2} = \{ \gamma^{1/2} : \gamma \in \Gamma \}$.

\begin{defn}[Set-dilation of Bohr sets]
If $B = \Bohr(\Gamma, \rho)$ is a Bohr set, we write $2 \cdot B$ for the Bohr set $\Bohr(\Gamma^{1/2}, \rho)$.
\end{defn}
Note that this is compatible with the notation for set-dilation: $2 \cdot B = \{ 2x : x \in B \}$.

\begin{lemma}\label{lemma:bohrDilates}
If $B$ is a Bohr set and $\tau \geq 0$, then
\[ (2 \cdot B)_\tau = 2 \cdot (B_\tau). \]
In particular, if $B$ is regular, then so is $2\cdot B$.
\end{lemma}

We shall use the following almost-periodicity result for convolutions that works relative to Bohr sets. While it does not explicitly appear in the literature, it is not a far cry from the combination of the almost-periodicity ideas of \cite{CrSi:2010} with the Chang--Sanders lemma on large spectra as in \cite{CrLaSi:2013, ScSi:2016}. The main differences are the presence of an $L^1$-norm (as opposed to an $L^0$-type estimate in \cite{CrSi:2010}) and that the $L^p$-norms are restricted to a Bohr set. We delay the proof of this (and some generalisations) to Section \ref{section:LpMeasures}.

\begin{theorem}[$L^p$-almost-periodicity relative to a Bohr set]\label{thm:LpBohr}
Let $m \geq 1$ and $\epsilon, \delta \in (0,1)$. Let $A, L$ be subsets of a finite abelian group $G$, with $\eta := \abs{A}/\abs{L} \leq 1$, and let $B \subset G$ be a regular Bohr set of rank $d$ and radius $\rho$. Suppose $\abs{A+S} \leq K\abs{A}$ for a subset $S \subset B_\tau$, where $B_\tau$ is regular and $\tau \leq (c\delta)^{2m}/d\log(2/\delta \eta)$. Then there is a regular Bohr set $T \leq B_\tau$ of rank at most $d + d'$ and radius at least $\rho \tau \delta \eta^{1/2}/d^2d'$, where
\[ d' \ll m\epsilon^{-2} \log^2(2/\delta\eta) \log(2K) + \log(1/\mu_{B_\tau}(S)), \]
such that, for each $t \in T$,	
\[ \norm{ \mu_A*1_L(\cdot+t) - \mu_A*1_L }_{L^{2m}(B)} \leq \epsilon \norm{f}_{L^m(B)}^{1/2} + \epsilon^{2-1/m} \norm{f}_{L^1(B)}^{1/2m} + \delta. \]
In particular, 
\[ \norm{ \mu_A*1_L*\mu_T - \mu_A*1_L }_{L^{2m}(B)} \leq \epsilon \norm{f}_{L^m(B)}^{1/2} + \epsilon^{2-1/m} \norm{f}_{L^1(B)}^{1/2m} + \delta. \]
\end{theorem}

\section{The main argument}
We can now describe the main argument. As mentioned in the previous section, we shall work with a pair $(B, B')$ of Bohr sets, regularity ensuring that $B+B' \approx B$. We shall correspondingly have a pair $(A, A')$ of sets, with $A \subset B$ and $2\cdot A' \subset B'$, each of relative density at least $\alpha$. There will then be two cases:

\begin{itemize}[leftmargin=2em]
\item If $\norm{ \mu_A*1_A }_{L^{2m}(B')} \geq 10 \alpha$, then we apply $L^{2m}(B')$-almost-periodicity to get that $\norm{ \mu_A*1_A*\mu_T }_{L^{2m}(B')}$ is large for some Bohr set $T$, from which a density increment is immediate.
\item If $\norm{ \mu_A*1_A }_{L^{2m}(B')} \leq 10 \alpha$, then the $L^{4m}(B')$-almost-periodicity result is particularly efficient, giving a large Bohr set $B$ such that $\langle \mu_A*1_A*\mu_T, \mu_{2\cdot A'} \rangle \approx \langle \mu_A*1_A, \mu_{2\cdot A'} \rangle$. Assuming that the number of 3APs across $(A,A',A)$ is small, say $\langle \mu_A*1_A, \mu_{2\cdot A'} \rangle \leq \tfrac{1}{4}\alpha$, this tells us that the same thing is true with an extra convolution with $\mu_T$, which quickly leads to a density increment.
\end{itemize}

\subsection*{Large $L^p$-norm of convolution implies density increment}
Here we expand upon the first case above, namely the one in which
\[ \norm{ \mu_A*1_A }_{L^{2m}(B')} \geq 10 \alpha. \]

\begin{proposition}
Let $G$ be a finite abelian group of odd order, let $B \subset G$ be a regular Bohr set, and let $B' \leq 2\cdot B$ be regular of rank $d$ and radius $\rho$. If $A \subset B$ is a set of relative density at least $\alpha$ with
\[ \norm{ \mu_A*1_A }_{L^{2m}(B')} \geq 10 \alpha \]
for some $m \in \N$, then there is a regular Bohr set $T \leq B'$ of rank at most $d+d'$ and radius at least $\rho \alpha^{Cm}/d^3$, where $d' \ll m \alpha^{-1} \log(2/\alpha)^3$, such that $\norm{ 1_A*\mu_T }_\infty \geq 5 \alpha$.
\end{proposition}
\begin{proof}
Let $\epsilon = c\alpha^{1/2}$, $\delta = c\alpha$ and apply Theorem \ref{thm:LpBohr} with these parameters to the convolution $\mu_A*1_A$, with the Bohr set $B'$ in place of $B$, and $\tau = (c\alpha)^{Cm}/d$ chosen so that $S := B'_\tau$ is regular. We then have that
\[ \abs{A+S} \leq \abs{B + B'_\tau} \leq \abs{B + B_{2\tau}} \leq \abs{B_{1+2\tau}} \leq 2\abs{B} \leq \tfrac{2}{\alpha}\abs{A}, \]
by Lemma \ref{lemma:bohrDilates} and regularity, allowing us to take $K = 2/\alpha$. This gives us a Bohr set $T \leq B'$ of the required rank and radius such that
\[ \norm{ \mu_A*1_A*\mu_T - \mu_A*1_A }_{L^{2m}(B')} \leq \epsilon \norm{\mu_A*1_A}_{L^m(B')}^{1/2} + \epsilon^{2-1/m} \norm{\mu_A*1_A}_{L^1(B')}^{1/2m} + \delta. \]
Now, we may assume that $\norm{\mu_A*1_A}_{L^1(B')} = \mu_A*1_A*\mu_{B'}(0) < 5\alpha$, as otherwise we are done (with $T = B'$). Thus
\[ \norm{ \mu_A*1_A*\mu_T }_{L^{2m}(B')} \geq \norm{ \mu_A*1_A }_{L^{2m}(B')} - \epsilon \norm{\mu_A*1_A}_{L^m(B')}^{1/2} - \epsilon^{2-1/m} (5\alpha)^{1/2m} - \delta. \]
By nesting of $L^p$-norms, the right-hand side here is at least
\begin{align*}
\norm{ \mu_A*1_A }_{L^{2m}(B')}^{1/2}&\left( \norm{ \mu_A*1_A }_{L^{2m}(B')}^{1/2} - \epsilon \right) - \epsilon^2 (5\alpha/\epsilon^2)^{1/2m} - \delta \\
&\geq (10 - c\sqrt{10} - c\sqrt{5} - c)\alpha,
\end{align*}
by our choice of $\epsilon$ and $\delta$. Thus, provided the constants in these parameters are chosen appropriately, we are done, as $\norm{\mu_A*1_A*\mu_T}_{L^{2m}(B')} \leq \norm{1_A*\mu_T}_\infty$.
\end{proof}

\subsection*{Small $L^p$-norm of convolution and few 3APs implies density increment}

Here we expand upon how to argue in the case 
\[ \norm{ \mu_A*1_A }_{L^{2m}(B')} \leq 10 \alpha. \]

\begin{proposition}
Let $G$ be a finite abelian group of odd order, let $B \subset G$ be a regular Bohr set, and let $B'$ be a regular Bohr set of rank $d$ and radius $\rho$ with $B' \subset B_{c/\rk(B)}$. Let $A \subset B$ and $2 \cdot A' \subset B'$ be sets of relative densities at least $\alpha$. If
\[ \norm{ \mu_A*1_A }_{L^{2m}(B')} \leq 10 \alpha \]
for some $m \geq C \log(2/\alpha)$, then either
\begin{enumerate}[leftmargin=2em,topsep=0pt]
\item (Many 3APs) $T(A,A',A) \geq \tfrac{1}{4} \alpha \abs{A}\abs{A'}$, or
\item (Density increment) there is a regular Bohr set $T \leq B'$ of rank at most $d+Cm\alpha^{-1} \log(2/\alpha)^3$, and radius at least $c\rho \alpha^{Cm}/d^3$, such that $\norm{ 1_A*\mu_T }_\infty \geq \tfrac{3}{2} \alpha$.
\end{enumerate}
\end{proposition}
\begin{proof}
Either we are in the first case of the proposition, or 
\[ \langle \mu_A*1_A, \mu_{2\cdot A'} \rangle \leq \tfrac{1}{4} \alpha. \]
We now apply Theorem \ref{thm:LpBohr} to $\mu_A*1_A$ with parameters $2m$, $\epsilon = c\alpha^{1/2}$, $\delta = c \alpha$, the Bohr set $B'$ in place of $B$, and $S = B'_\tau$ with $\tau = (c\alpha)^{Cm}/d$, giving us a Bohr set $T \leq B'_\tau$ of the required rank and radius such that 
\[ \norm{ \mu_A*1_A*\mu_T*\mu_T - \mu_A*1_A }_{L^{4m}(B')} \leq \epsilon \norm{ \mu_A*1_A }_{L^{2m}(B')}^{1/2} + \epsilon^{2-1/2m} \norm{ \mu_A*1_A }_{L^1(B')}^{1/4m} + \delta. \]
By assumption and choice of parameters, and assuming that $\norm{\mu_A*1_A}_{L^1(B')} \leq \tfrac{3}{2} \alpha$ (or else increment) as in the previous argument, we thus have that
\[ \norm{ \mu_A*1_A*\mu_T*\mu_T - \mu_A*1_A }_{L^{4m}(B')} \leq c \alpha, \]
where the positive constant $c$ may be chosen as small as we wish. Thus, letting $q$ be such that $1/q + 1/4m = 1$, H\"older's inequality yields
\begin{align*}
\abs{ \langle \mu_A*1_A*\mu_T*\mu_T, \mu_{2\cdot A'} \rangle  -&  \langle \mu_A*1_A, \mu_{2\cdot A'} \rangle } \\
&\leq \tfrac{1}{\mu_{B'}(2\cdot A')} \norm{ 1_{2\cdot A'} }_{L^q(B')} \norm{ \mu_A*1_A*\mu_T*\mu_T - \mu_A*1_A }_{L^{4m}(B')} \\
&\leq \mu_{B'}(2\cdot A')^{-1/4m} c \alpha \\
&\leq c \alpha^{1-1/4m}.
\end{align*}
Since $m \geq C \log(2/\alpha)$, this is at most $2c \alpha$. Picking $c$ small enough thus gives that 
\[ \langle \mu_A*1_A*\mu_T*\mu_T, \mu_{2\cdot A'} \rangle \leq \tfrac{1}{2} \alpha. \]
There is thus some $x \in 2\cdot A' \subset B' \subset B_{c/\rk(B)}$ such that 
\[ \mu_A*1_A*\mu_T*\mu_T(x) \leq \tfrac{1}{2}\alpha. \]
We are then done by the following lemma.
\end{proof}

\begin{lemma}
Let $B \subset G$ be a regular Bohr set and let $A \subset B$ be a set of relative density $\alpha> 0$. Let $\lambda \in [0,1]$, and suppose $T \subset B_\tau$ where $\tau \ll \lambda^2/\rk(B)$. If
\[ \mu_A*1_A*\mu_T*\mu_T(x) \leq (1-2\lambda^2)\alpha \]
for some $x \in B_\tau$, then $\norm{1_A*\mu_T}_\infty \geq (1+\lambda)\alpha$.
\end{lemma}
\begin{proof}
Suppose $\norm{1_A*\mu_T}_\infty \leq (1+\lambda)\alpha$. Let $F = \frac{1_A*\mu_T}{(1+\lambda)\alpha}$, so that $0 \leq F \leq 1_{B_{1+\tau}}$. In particular, we have the pointwise inequality
\[ 0 \leq (1_{B_{1+\tau}} - F)*(1_{B_{1+\tau}} - F) = F*F - 2 F*1_{B_{1+\tau}} + 1_{B_{1+\tau}}*1_{B_{1+\tau}}. \]
Thus
\begin{equation}\label{eqn:convptwise}
F*F(x) \geq 2 F*1_{B_{1+\tau}}(x) - 1_{B_{1+\tau}}*1_{B_{1+\tau}}(x)
\end{equation}
for every $x$. We now use regularity to estimate the right-hand side for $x \in B_\tau$. Indeed,
\[ \abs{F*1_{B_{1+\tau}}(x) - F*1_{B_{1+\tau}}(0)} \leq \norm{F}_\infty \sum_y \abs{1_{B_{1+\tau}}(y-x) - 1_{B_{1+\tau}}(y) } \leq \abs{B_{1+2\tau}\setminus B} \ll \tau d \abs{B}, \]
where $d := \rk(B)$, since $B$ is regular, and furthermore
\[ F*1_{B_{1+\tau}}(0) = \sum F = \abs{B}/(1+\lambda). \]
The second term in \eqref{eqn:convptwise} can be bounded trivially: 
\[ 1_{B_{1+\tau}}*1_{B_{1+\tau}}(x) \leq \abs{B_{1+\tau}} \leq (1+c\tau d)\abs{B}, \]
again by regularity. Renormalising \eqref{eqn:convptwise} and picking the implied constant in the bound for $\tau$ in the hypothesis small enough, we thus have
\[ \mu_A*1_A*\mu_T*\mu_T(x) \geq \left(2(1+\lambda) - (1+c\lambda^2)(1+\lambda)^2\right)\alpha, \]
where $c > 0$ is as small a fixed constant as we like. Picking $c = 1/2$, say, makes this bigger than $(1-2\lambda^2)\alpha$, as desired. 
\end{proof}
\begin{remark}
There are several variants of this type of result, converting deviations to increments. Perhaps the most standard one uses Fourier analysis, which gives a slightly better $\lambda$-dependence, but this is of no relevance in our application.
\end{remark}

\subsection*{The iteration}
Combining the previous two propositions immediately yields the following.

\begin{proposition}\label{prop:twoSetIterator}
Let $G$ be a finite abelian group of odd order, let $B \subset G$ be a regular Bohr set, and let $B' \leq 2\cdot B$ be regular of rank $d$ and radius $\rho$ with $B' \subset B_{c/\rk(B)}$. Let $A \subset B$ and $2 \cdot A' \subset B'$ be sets of relative densities at least $\alpha$. Then either
\begin{enumerate}[leftmargin=2em,topsep=0pt]
\item (Many 3APs) $T(A,A',A) \geq \tfrac{1}{4} \alpha \abs{A}\abs{A'}$, or
\item (Density increment) there is a regular Bohr set $T \leq B'$ of rank at most $d + C\alpha^{-1} \log(2/\alpha)^4$, and radius at least $c\rho \alpha^{C\log(2/\alpha)}/d^3$, such that $\norm{ 1_A*\mu_T }_\infty \geq \tfrac{3}{2} \alpha$.
\end{enumerate}
\end{proposition}

If not for the fact that we need to work with the two copies of the set $A$ here, one living in a slightly narrower Bohr set than the other, we could just iterate this proposition to yield the theorem. This is where the following `two scales' lemma of Bourgain's \cite{Bo:1999} comes in: it converts a single set $A$ in a Bohr set to two copies of roughly the original density living inside narrower Bohr sets (or else we have a density increment). The lemma is now fairly standard, but we include the proof for completeness.

\begin{lemma}\label{lemma:twoscales}
Let $B$ be a regular Bohr set of rank $d$, let $A \subset B$ have relative density at least $\alpha$, and let $B', B'' \subset B_{c\alpha/d}$. Then either 
\begin{enumerate}[topsep=0pt]
\item there is an $x \in B$ such that $1_A*\mu_{B'}(x) \geq \tfrac{3}{4}\alpha$ and $1_A*\mu_{B''}(x) \geq \tfrac{3}{4}\alpha$, or
\item $\norm{1_A*\mu_{B'}}_\infty \geq \tfrac{9}{8}\alpha$ or $\norm{1_A*\mu_{B''}}_\infty \geq \tfrac{9}{8}\alpha$.
\end{enumerate}
\end{lemma}
\begin{proof}
Picking the constant $c$ in the radius-narrowing small enough, regularity yields
\[ \abs{ 1_A*\mu_B*\mu_{B'}(0) - 1_A*\mu_B(0) } \leq \tfrac{1}{\abs{B}} \E_{t \in B'} \sum_x \abs{1_B(x+t) - 1_B(x)} \leq \tfrac{1}{16} \alpha, \]
and similarly for $B''$. Since $1_A*\mu_B(0) = \mu_B(A) = \alpha$, this implies that 
\[ \E_{x \in B}\, \big(1_A*\mu_{B'}(x) + 1_A*\mu_{B''}(x) \big) \geq (2 - \tfrac{1}{8}) \alpha, \]
and so there exists $x \in B$ such that $1_A*\mu_{B'}(x) + 1_A*\mu_{B''}(x) \geq (2 - \frac{1}{8}) \alpha$. With such an $x$, if we are not in the second case of the conclusion then  
\[ 1_A*\mu_{B'}(x) \geq (2- \tfrac{1}{8})\alpha - \tfrac{9}{8}\alpha = \tfrac{3}{4}\alpha, \]
and similarly for $B''$, and so we are done. 
\end{proof}

\begin{proposition}[Main iterator]\label{prop:mainIterator}
Let $G$ be a finite abelian group of odd order, let $B \subset G$ be a regular Bohr set rank $d$ and radius $\rho$, and let $A \subset B$ be a set of relative density at least $\alpha$. Then either
\begin{enumerate}[topsep=0pt]
\item (Many 3APs) $T(A) \geq \exp\left(-Cd\log(d/\alpha)\right)\abs{A}^2$, or
\item (Density increment) there is a regular Bohr set $T \leq B$ of rank at most $d + C\alpha^{-1} \log(2/\alpha)^4$, and radius at least $c\rho \alpha^{C\log(2/\alpha)}/d^5$, such that $\norm{ 1_A*\mu_T }_\infty \geq \tfrac{9}{8} \alpha$.
\end{enumerate}
\end{proposition}
\begin{proof}
Increasing $\alpha$ if necessary, we may assume that $\mu_B(A) = \alpha$. Let $B^{(1)} = B_{c\alpha/d}$ and $B^{(2)} = B_{c/d}^{(1)}$, with small constants $c$ picked so that these are regular. Applying Lemma \ref{lemma:twoscales} with these sets, we are either done, obtaining a density increment with $T$ being $B^{(1)}$ or $B^{(2)}$, or else we find an $x$ such that $1_A*\mu_{B^{(i)}}(x) \geq \tfrac{3}{4}\alpha$ for $i=1,2$. In the latter case, we define $A^{(i)} = (A-x) \cap B^{(i)}$, so that $\mu_{B^{(i)}}(A^{(i)}) \geq \tfrac{3}{4}\alpha$,  
and, moreover by Lemma \ref{lemma:bohrDilateGrowth},
\[ \abs{A^{(1)}} \gg \left(\frac{c\alpha}{d}\right)^{3d} \abs{A} \quad \text{and} \quad \abs{A^{(2)}} \gg \left(\frac{c\alpha}{d^2}\right)^{3d} \abs{A}. \]
Note that by translation-invariance of three-term progressions,
\[ T(A) \geq T\left(A^{(1)}, A^{(2)}, A^{(1)}\right), \]
and if this quantity is at least $\tfrac{3}{16} \alpha\abs{A^{(1)}} \abs{A^{(2)}}$ then we are in the first case of the conclusion. If not, apply Proposition \ref{prop:twoSetIterator} with $B^{(1)}$ in place of $B$, $B' = 2 \cdot B^{(2)}$, which is regular by Lemma \ref{lemma:bohrDilates}, and $A^{(1)}$, $A^{(2)}$ in place of $A$, $A'$, respectively. We must then be in the second case of the conclusion of that lemma, giving us the Bohr set $T$ required in the conclusion, since
\[ \norm{1_A*\mu_T}_\infty \geq \norm{1_{A{^{(1)}}}*\mu_T}_\infty \geq \tfrac{3}{2} \cdot \tfrac{3}{4}\alpha = \tfrac{9}{8}\alpha. \qedhere \]
\end{proof}

It is now straightforward to iterate this to prove our main theorem.

\begin{theorem}
Let $G$ be a finite abelian group of odd order, and let $A \subset G$ be a set of density at least $\alpha$. Then
\[ T(A) \geq \exp\left(-C \alpha^{-1} \log(1/\alpha)^C\right)\abs{A}^2. \]
\end{theorem}
\begin{proof}
We define a sequence of Bohr sets $B^{(i)}$ of rank $d_i$ and radius $\rho_i$, and corresponding subsets $A^{(i)}$ of relative densities $\alpha_i$, starting with $B^{(0)} = \Bohr(\{1\}, 2) = G$ and $A^{(0)} = A$. Having defined $B^{(i)}$ and $A^{(i)}$, we apply Proposition \ref{prop:mainIterator} to these sets. If we are in the first case of the conclusion, we exit the iteration, and if we are in the second case, say with $1_{A^{(i)}}*\mu_T(x) \geq \tfrac{9}{8}\alpha_i$, we define $B^{(i+1)} = T$ and $A^{(i+1)} = (A^{(i)}-x) \cap T$. We thus have
\[ d_{i+1} \leq d_i + C \alpha_i^{-1}\log(2/\alpha)^4, \qquad \rho_{i+1} \gg \rho_i \alpha^{C\log(2/\alpha)}/d_i^5, \qquad \alpha_{i+1} \geq \tfrac{9}{8} \alpha_i. \]
Since the densities are increasing exponentially and can never be bigger than $1$, the procedure must terminate with some set $A^{(k)}$ with $k \ll \log(1/\alpha)$. By summing the geometric progression, the final rank satisfies $d_k \ll \alpha^{-1} \log(2/\alpha)^4$, and the final radius satisfies $\rho_k \geq \exp\left(-C \log(2/\alpha)^3 \right)$. Having exited the iteration, we thus have
\[ T(A) \geq T\left(A^{(k)}\right) \geq \exp\left(-C d_k \log(d_k/\alpha) \right) \abs{A^{(k)}}^2 \geq \exp\left(-C \alpha^{-1} \log(2/\alpha)^7 \right)\abs{A}^2, \]
by Lemma \ref{lemma:bohrDilateGrowth}, as desired.
\end{proof}

\section{$L^p$-almost-periodicity with more general measures}\label{section:LpMeasures}
In this section we record some results on the $L^p$-almost-periodicity of convolutions, including a proof of Theorem \ref{thm:LpBohr}. These results have their origins in \cite{CrSi:2010}, but since we require a couple of slight twists in the fundamentals of the arguments, we give an essentially self-contained treatment. Our presentation is at a somewhat greater level of generality than needed for the current application; we expect this to be useful for future applications, however, as well as being conceptually illuminating, perhaps. The first few results are phrased in terms of an arbitrary group $G$, which we view as a discrete group with the discrete $\sigma$-algebra when discussing measures.\footnote{It is clear that everything extends naturally to locally compact groups, but we have no need for this generality here.}

Thus when we work with $L^p$ norms restricted to some measure $\mu$ on $G$, we have
\[ \norm{f}_{L^p(\mu)}^p = \sum_{x} \mu(x)\abs{f(x)}^p. \]
We take as our definition of convolution
\[ f*g(x) = \sum_y f(y) g(y^{-1} x), \]
and, for a $k$-tuple $\tup{a} = (a_1, \ldots, a_k)$, we write $\mu_\tup{a} = \E_{j \in [k]} 1_{\{a_j\}}$.

The following moment-type estimates were essentially proved in \cite{CrSi:2010}.

\begin{lemma}\label{lemma:probsampMeasure}
Let $m, k \geq 1$. Let $A, L$ be finite subsets of a group $G$, let $\mu$ be a measure on $G$, and denote 
\[ f = \mu_A*1_L \cdot (1 - \mu_A*1_L). \]
If $\tup{a} \in A^k$ is sampled uniformly at random, then, provided $k \geq Cm/\epsilon^2$,
\[ \E \norm{ \mu_\tup{a}*1_L - \mu_A*1_L }_{L^{2m}(\mu)}^{2m} \leq \epsilon^{2m} \norm{f}_{L^m(\mu)}^m + \epsilon^{4m-2} \norm{f}_{L^1(\mu)}. \]
\end{lemma}

We include a proof in Appendix \ref{app:moments} in order to cater for the differences from \cite{CrSi:2010}.

\begin{defn}[Translation operator]
Given a function $f$ on a group $G$, and an element $t \in G$, we write $\tau_t f$ for the function on $G$ defined by
\[ \tau_t f(x) = f(tx). \]
Similarly, if $\mu$ is a measure on $G$, we write $\tau_t \mu$ for the measure given by $\tau_t \mu(X) = \mu(tX).$ Thus
\[ \E_{x \sim \tau_t \mu} f(x) = \E_{x \sim \mu} f(t^{-1}x). \]
\end{defn}

\begin{defn}
Let $\nu, \mu$  be two measures on a group $G$. We say that $\nu \leq \mu$ if $\nu(X) \leq \mu(X)$ for every measurable $X$, that is, if
\[ \E_{\nu} f \leq \E_{\mu} f \]
for every integrable $f \geq 0$.
\end{defn}

\begin{defn}[$S$-invariant pairs of measures]
Let $\nu, \mu$ be two measures on a group $G$, and let $S \subset G$. We say that $(\nu, \mu)$ is \emph{$S$-invariant if} $\tau_t \nu \leq \mu$ for every $t \in S$.
\end{defn}

A prototypical example is the pair $(1_{B_{1-\tau}}, 1_B)$ for a Bohr set $B$, which is $B_\tau$-invariant. Of course the pair $(1_G, 1_G)$ is $G$-invariant. (Here $1_X(A) = \abs{A\cap X}$.)

In the following proof, if $X$ is a subset of a group then we write $X^{\otimes k}$ for the $k$th Cartesian power of $X$, in order to distinguish it from the product set $X^k = X\cdot X \cdots X$.

\begin{theorem}\label{thm:plainAPwMeasures}
Let $m, n \geq 1$, $\epsilon \in (0,1)$. Let $A, L, S$ be finite subsets of a group $G$, and suppose $(\nu, \mu)$ is an $(S^{-1}S)^n$-invariant pair of measures on $G$. Suppose $\abs{S\cdot A} \leq K\abs{A}$. Then there is a subset $T \subset S$, $\abs{T} \geq 0.99K^{-Cm n^2/\epsilon^2}\abs{S}$, such that, for every $t \in (T^{-1}T)^n$,
\[ \norm{ \tau_t (\mu_A*1_L) - \mu_A*1_L }_{L^{2m}(\nu)} \leq \epsilon \norm{f}_{L^m(\mu)}^{1/2} + \epsilon^{2-1/m} \norm{f}_{L^1(\mu)}^{1/2m}/n^{1-1/m}. \]
\end{theorem}
The main differences between this and the results in \cite{CrSi:2010} lie in the restriction of the norms and in the slight extra care to give an $L^1$-norm rather than an $L^0$-type estimate.
\begin{proof}
Let $\epsilon_0 = \epsilon/2n$. By Lemma \ref{lemma:probsampMeasure} applied with $k = Cm/\epsilon_0^2$, we get that if $\tup{a} \in A^{\otimes k}$ is sampled uniformly then with probability at least $0.99$,
\[ \norm{ \mu_\tup{a}*1_L - \mu_A*1_L }_{L^{2m}(\mu)} \leq \epsilon_0 \norm{f}_{L^m(\mu)}^{1/2} + \epsilon_0^{2-1/m} \norm{f}_{L^1(\mu)}^{1/2m}. \]
Let us call tuples $\tup{a} \in A^{\otimes k}$ satisfying this bound \emph{good}, so that
\[ \P_{\tup{a} \in A^{\otimes k}}( \text{$\tup{a}$ is good} ) \geq 0.99. \]
Now let us write $\Delta(S) = \{(t,\ldots,t) \in S^{\otimes k}\}$, and let us identify elements $t \in S$ with the corresponding tuple in $\Delta(S)$. Define, for each $\tup{a} \in \Delta(S)\cdot A^{\otimes k}$,
\[ T_\tup{a} = \{ t \in S : \text{$t^{-1}\tup{a}$ is good} \} \subset S. \]
We now claim two things: firstly, that $(T_\tup{a}^{-1}\cdot T_\tup{a})^n$ is a set of almost-periods for any $\tup{a}$; secondly, that $\abs{T_\tup{a}}$ is large on average. We begin with the second claim: for each $t \in S$,
\begin{align*}
\P_{\tup{a} \in \Delta(S)\cdot A^{\otimes k}}( \text{$t^{-1} \tup{a}$ is good} ) &= \P_{\tup{a} \in t^{-1} \Delta(S)\cdot A^{\otimes k}}( \text{$\tup{a}$ is good} ) \\
& \geq \frac{\abs{A}^k}{\abs{\Delta(S)\cdot A^{\otimes k}}} \P_{\tup{a} \in A^{\otimes k}}( \text{$\tup{a}$ is good} ) \\
& \geq 0.99 K^{-k},
\end{align*}
since $\Delta(S) \cdot A^{\otimes k} \subset (S \cdot A)^{\otimes k}$, and so 
\[ \E_{\tup{a} \in \Delta(S) \cdot A^{\otimes k}} \abs{T_\tup{a}} = \sum_{t \in S} \P_{\tup{a} \in \Delta(S) \cdot A^{\otimes k}}( \text{$t^{-1}\tup{a}$ is good} ) \geq 0.99 K^{-k} \abs{S}. \]
This was the second claim; we turn now to showing the first.

Fix any $\tup{a}$ and let $T = T_\tup{a}$, and for brevity write $g = \mu_A*1_L$. Then, by definition, for $t \in T$ we have
\begin{equation}
\norm{ \tau_t(\mu_\tup{a}*1_L) - g }_{L^{2m}(\mu)} \leq \epsilon_0 \norm{f}_{L^m(\mu)}^{1/2} + \epsilon_0^{2-1/m} \norm{f}_{L^1(\mu)}^{1/2m}. \label{eqn:goodtranslatebound}
\end{equation}
Now let $t_1, \ldots, t_n \in T^{-1} T$. Then
\begin{align*}
\norm{ \tau_{t_1 \cdots t_n}g - g }_{L^{2m}(\nu)} &\leq \norm{ \tau_{t_1 \cdots t_n}g - \tau_{t_n} g }_{L^{2m}(\nu)} + \norm{ \tau_{t_n}g - g }_{L^{2m}(\nu)} \\
&= \norm{ \tau_{t_1 \cdots t_{n-1}}g - g }_{L^{2m}(\tau_{t_n^{-1}}\nu)} + \norm{ \tau_{t_n}g - g }_{L^{2m}(\nu)}.
\end{align*}
Carrying on in this way, we have
\begin{equation}
\norm{ \tau_{t_1 \cdots t_n}g - g }_{L^{2m}(\nu)} \leq \norm{ \tau_{t_1} g - g }_{L^{2m}(\tau_{r_1} \nu)} + \cdots + \norm{ \tau_{t_n} g - g }_{L^{2m}(\tau_{r_n} \nu)}, \label{eqn:triinvariance}
\end{equation}
where $r_j \in (T^{-1}T)^{n-j}$. Consider one of the summands here, with $r = r_j$ and $t = t_j = s_1^{-1} s_2$ for some elements $s_i \in T$. We have
\[ \norm{ \tau_t g - g }_{L^{2m}(\tau_r \nu)} \leq \norm{ \tau_{s_1^{-1} s_2} g - \tau_{s_2}(\mu_\tup{a}*1_L) }_{L^{2m}(\tau_r \nu)} + \norm{ \tau_{s_2}(\mu_\tup{a}*1_L) - g }_{L^{2m}(\tau_r \nu)}. \]
The first term here equals
\[ \norm{ g - \tau_{s_1}(\mu_\tup{a}*1_L) }_{L^{2m}(\tau_{r s_2^{-1} s_1} \nu)}, \]
and so, since $T \subset S$ and $(\nu, \mu)$ is $(S^{-1}S)^n$-invariant, both of these terms can be bounded as in \eqref{eqn:goodtranslatebound}. Thus
\[ \norm{ \tau_{t_1 \cdots t_n}g - g }_{L^{2m}(\nu)} \leq 2n\left(\epsilon_0 \norm{f}_{L^m(\mu)}^{1/2} + \epsilon_0^{2-1/m} \norm{f}_{L^1(\mu)}^{1/2m} \right), \]
which proves the claim that the set $(T^{-1}T)^n$ is a set of almost-periods for $\mu_A*1_L$.

Letting $\tup{a}$ be some tuple for which $T = T_\tup{a}$ has size at least $0.99K^{-k} \abs{S}$ yields the theorem.
\end{proof}

We now bootstrap this in a standard way using Fourier analysis, making use of the following local version of Chang's lemma on large spectra due to Sanders \cite{Sa:2008}.

\begin{lemma}[Chang--Sanders]\label{lemma:localChang}
Let $\delta, \nu \in (0, 1]$. Let $G$ be a finite abelian group, let $B = \Bohr(\Gamma, \rho) \subset G$ be a regular Bohr set of rank $d$ and let $X \subset B$. Then there is a set of characters $\Lambda \subset \Ghat$ and a radius $\rho'$ with
\[ \abs{\Lambda} \ll \delta^{-2} \log(2/\mu_B(X)) \quad \text{and} \quad \rho' \gg \rho \nu \delta^2/d^2\log(2/\mu_B(X)) \]
such that
\[ \abs{1 - \gamma(t)} \leq \nu \quad \text{for all $\gamma \in \Spec_\delta(\mu_X)$ and $t \in \Bohr(\Gamma \cup \Lambda, \rho')$}. \]
\end{lemma}

\begin{theorem}[$L^p$-almost-periodicity relative to Bohr-compatible measures]\label{thm:Lp-ap_BohrMeasures}
Let $m \geq 1$ and $\epsilon, \delta \in (0,1)$. Let $A, L$ be subsets of a finite abelian group $G$ with $\eta := \abs{A}/\abs{L} \leq 1$, let $B \subset G$ be a regular Bohr set of rank $d$ and radius $\rho$, and let $(\nu, \mu)$ be an $rB$-invariant pair of measures on $G$, where $r \geq C\log(2/\delta \eta)$. Suppose $\abs{A+S} \leq K\abs{A}$ for a subset $S \subset B$. Then there is a regular Bohr set $B' \leq B$ of rank at most $d + d'$ and radius at least $\rho \delta \eta^{1/2}/d^2d'$, where
	\[ d' \ll m\epsilon^{-2} \log^2(2/\delta\eta) \log(2K) + \log(1/\mu_B(S)), \]
such that, for each $t \in B'$,	
\[ \norm{ \mu_A*1_L(\cdot+t) - \mu_A*1_L }_{L^{2m}(\nu)} \leq \epsilon \norm{f}_{L^m(\mu)}^{1/2} + \epsilon^{2-1/m} \norm{f}_{L^1(\mu)}^{1/2m} + \delta \norm{\nu}_{\ell^1}^{1/2m}. \]
\end{theorem}
\begin{proof}
We could deduce a version of this from Theorem \ref{thm:plainAPwMeasures} as stated, working with an intermediate measure $\nu_2$ for which $(\nu, \nu_2)$ and $(\nu_2, \mu)$ are invariant, but for a cleaner statement we instead argue directly, picking up where the proof of that theorem left off. Indeed, say we have followed that argument with parameters $m$, $n = \floor{(r-1)/2}$ and $\epsilon/2$, thus obtaining a set $T \subset S$ with
\[ \mu_B(T) \geq 0.99K^{-Cm r^2/\epsilon^2} \mu_B(S) \]
such that, for each $s \in nT-nT$,
\[ \norm{ \tau_s g - g }_{L^{2m}(\nu)} \leq \epsilon' := \tfrac{1}{2}\epsilon \norm{f}_{L^m(\mu)}^{1/2} + \tfrac{1}{2}\epsilon^{2-1/m} \norm{f}_{L^1(\mu)}^{1/2m}, \]
where again $g = \mu_A*1_L$. Let us then write $\sigma = \mu_T^{(n)}*\mu_{-T}^{(n)}$, where $\mu_X^{(n)}$ represents the $n$-fold convolution $\mu_X*\cdots*\mu_X$. By the triangle inequality, we then have
\[ \norm{ g*\sigma - g }_{L^{2m}(\nu)} \leq \E_{t_j \in T} \norm{ \tau_{s} g - g }_{L^{2m}(\nu)} \leq \epsilon', \]
where we have written $s = t_1+\cdots+t_n-t_{n+1}-\cdots-t_{2n}$ in the expectation.
We also want this estimate to hold for any translate $\tau_t \nu$ of $\nu$ with $t \in B$, which follows from $(\nu,\mu)$ being $(2n+1)B$-invariant: for any $t_1,\ldots,t_n \in T-T$ and $t \in B$, the bound \eqref{eqn:triinvariance} holds with $\nu$ replaced by $\tau_{-t}(\nu)$, and the final measures appearing thereafter in the proof are still dominated by $\mu$, by $(2n+1)B$-invariance, meaning that also
\[ \norm{ \tau_t(g*\sigma) - \tau_t g }_{L^{2m}(\nu)} \leq \epsilon' \]
holds for all $t \in B$.

Now we carry out the Fourier-bootstrapping in a standard way. By the triangle inequality, we have that, for any $t \in B$,
\begin{align*}
\norm{ \tau_t g - g }_{L^{2m}(\nu)} \leq \norm{\tau_t g - \tau_t(g*\sigma)}_{L^{2m}(\nu)} 
+ \norm{\tau_t(g*\sigma) - g*\sigma}_{L^{2m}(\nu)} 
+ \norm{g*\sigma - g}_{L^{2m}(\nu)},
\end{align*}
which, by the above, is at most 
\[ 2\epsilon' + \norm{\tau_t(g*\sigma) - g*\sigma}_{L^{2m}(\nu)}. \]
The last term here is at most
\begin{align*}
\norm{\nu}_{\ell^1}^{1/2m} \norm{\tau_t(g*\sigma) - g*\sigma}_{L^\infty(G)},
\end{align*}
and it is in bounding this that we shall need to pick $t$ carefully. Indeed, apply Lemma \ref{lemma:localChang} to $T \subset B$ with parameter $\delta = 1/2$ to get a regular Bohr set $B' \leq B$ of rank at most $d+d'$ and radius at least $\rho\delta\eta^{1/2}/d^2d'$, where
\[ d' \ll \log(2/\mu_B(T)) \ll m n^2\epsilon^{-2} \log(2K) + \log(1/\mu_B(S)) \]
such that 
\[ \abs{1-\gamma(t)} \leq \delta\eta^{1/2} \text{ for all } \gamma \in \Spec_{1/2}(\mu_T) \text{ and } t \in B'. \]
Taking $t \in B'$, then, we have by the Fourier inversion formula that
\begin{equation}
\norm{\tau_t(g*\sigma) - g*\sigma}_{L^\infty} \leq\ \E_{\gamma \in \Ghat}\, \abs{\widehat{\mu_A}(\gamma)} \abs{\widehat{1_L}(\gamma)} \abs{\widehat{\mu_T}(\gamma)}^{2n} \abs{\gamma(t) - 1},\label{eqn:FourierInfty}
\end{equation}
and we bound the terms in this average according to whether $\gamma \in \Spec_{1/2}(\mu_T)$ or not. If $\gamma \in \Spec_{1/2}(\mu_T)$ then $\abs{\gamma(t) - 1} \leq \delta\eta^{1/2}$, and if not then $\abs{\widehat{\mu_T}(\gamma)}^{2n} \leq 1/4^n \leq \delta\eta^{1/2}/2$, provided we pick $n = 2\ceiling{\log \delta^{-1}\eta^{-1}}$. Thus \eqref{eqn:FourierInfty} is at most twice
\[ \delta\,\E_{\gamma \in \Ghat}\, \abs{\widehat{\mu_A}(\gamma)} \abs{\widehat{1_L}(\gamma)}, \]
which, by Cauchy-Schwarz and Parseval's identity, is at most
\[ \delta\eta^{1/2} \E_{\gamma \in \Ghat}\, \abs{\widehat{\mu_A}(\gamma)} \abs{\widehat{1_L}(\gamma)} \leq \delta \eta^{1/2} \left( \E_{\gamma \in \Ghat}\, \abs{\widehat{\mu_A}(\gamma)}^2 \right)^{1/2} \left( \E_{\gamma \in \Ghat}\, \abs{\widehat{1_L}(\gamma)}^2 \right)^{1/2} = \delta, \]
recalling that $\eta = \abs{A}/\abs{L}$. Putting all these estimates together and replacing $\delta$ by $\delta/2$, we are done.
\end{proof}

The main almost-periodicity theorem used in this paper, Theorem \ref{thm:LpBohr}, is a simple corollary of this, using the regularity of Bohr sets through the following lemma. Using regularity at this point is somewhat inefficient quantitatively, adding an extra $\log\log$ to our final bound for Roth's theorem, but it allows for simpler statements.

\begin{lemma}\label{lemma:LpReg}
Let $B$ be a regular Bohr set of rank $d$, let $\delta \in [0,1]$, and suppose $\tau \leq c \delta^p/d$. Then, for any $F : G \to \C$ and $p \geq 1$,
\[ \norm{F}_{\ell^p(B)} \leq \norm{F}_{\ell^p(B_{1-\tau})} + \delta \norm{F}_{\ell^\infty(B)} \abs{B}^{1/p}. \]
\end{lemma}
\begin{proof}
By the triangle inequality
\[ \norm{F}_{\ell^p(B)}^p - \norm{F}_{\ell^p(B_{1-\tau})}^p \leq \norm{F}_{\ell^\infty(B)}^p \abs{B \setminus B_{1-\tau}}.\]
It follows from regularity that $\abs{B\setminus B_{1-\tau}}\ll \tau d\abs{B}$, and so the result follows if we choose $c$ small enough.
\end{proof}

It is now a short matter to deduce Theorem \ref{thm:LpBohr}, the almost-periodicity result with all the $L^p$-norms being relative to the same Bohr set.

\begin{proof}[Proof of Theorem \ref{thm:LpBohr}]
Let $r = \ceiling{C\log(2/\delta\eta)}$ and apply Theorem \ref{thm:Lp-ap_BohrMeasures} to $A$ and $L$ with parameters $m$, $\epsilon$, $\delta/2$, the Bohr set $B_\tau$ in place of $B$ and the $rB_\tau$-invariant pair of measures $\nu = 1_{B_{1-r\tau}}$, $\mu = 1_B$. This gives a Bohr set $T \leq B_\tau$ of the required rank and radius such that, for each $t \in T$,
\[ \norm{ \mu_A*1_L(\cdot+t) - \mu_A*1_L }_{\ell^{2m}(B_{1-r\tau})} \leq \epsilon \norm{f}_{\ell^m(B)}^{1/2} + \epsilon^{2-1/m} \norm{f}_{\ell^1(B)}^{1/2m} + \tfrac{1}{2}\delta\abs{B}^{1/2m}. \]
Since $\tau \leq c(\delta/2)^{2m}/dr$, the main claim follows from Lemma \ref{lemma:LpReg}. The `in particular' then follows by averaging and the triangle inequality.
\end{proof}

\section{Concluding remarks}
In some sense, it should not be altogether surprising that the almost-periodicity arguments of \cite{CrSi:2010} can be used to prove logarithmic bounds for Roth's theorem, as these results were used to reach this barrier in several other related problems, already in \cite{CrSi:2010} but also in \cite{CrLaSi:2013}. Being able to do this rests on using the more elaborate moment-bounds present in \cite{CrSi:2010} (or in this paper) for the random sampling, rather than the more usual Khintchine-type bounds.

\subsection*{The number of $\log\log$s}
The argument presented in this paper gives a bound of $r_3(N)/N \ll \tfrac{(\log\log N)^C}{\log N}$ with $C=7$. One of these $\log\log$s is caused by applying Bohr-set regularity to an $L^p$ norm with $p$ large, which makes for clean statements but is otherwise quite wasteful. Circumventing this and taking into account some further optimisations allows one to reduce this $C$, but not to below $4$, which is the best bound currently known \cite{Bl:2016}.

\appendix

\section{Almost-periodicity results}

The following result is \cite[Corollary 1.4]{CrSi:2010}
\begin{theorem}\label{thm:Lp-ap}
Let $p \geq 2$ and $\epsilon \in (0,1)$ be parameters. Let $G$ be a finite abelian group and let $A, L \subset G$ be finite subsets with $\abs{A} \geq \alpha \abs{G}$. Then there is a set $T \subset S$ with $\abs{T} \geq (\alpha/2)^{O(p\epsilon^{-2})} \abs{G}$ such that
\[ \norm{ \mu_A*1_L(\cdot+t) - \mu_A*1_L }_p \leq \epsilon \norm{ \mu_A*1_L }_{p/2}^{1/2} + \epsilon^2 \quad\text{for each $t \in T-T$.} \]
\end{theorem}
For completeness we include the following short deduction of the almost-periodicity result used in the finite field argument.

\begin{proof}[Proof of Theorem~\ref{thm:Lp}]
Let $k\geq 1$ be some parameter to be chosen later, and let $T$ be the set of almost-periods provided by Theorem~\ref{thm:Lp-ap}. It follows that
\[\norm{\mu_A\ast 1_L\ast \mu-\mu_A\ast 1_L}_{p}\leq k \epsilon \norm{\mu_A\ast 1_L}_{p/2}^{1/2}+k\epsilon^2,\]
where $\mu := \mu_{T-T}^{(k)}$ is the $k$-fold convolution $\mu_{T-T}*\cdots*\mu_{T-T}$. Thus, for any $t\in \mathbb{F}_q^n$, 
\begin{align*}
\norm{\mu_A\ast 1_L(\cdot +t)-\mu_A\ast 1_L}_p 
&\leq 2k\epsilon \norm{\mu_A\ast 1_L}_{p/2}^{1/2}+2k\epsilon^2 \\
&\quad+ \norm{\mu_A*1_L*\mu(\cdot+t) - \mu_A*1_B*\mu}_p.
\end{align*}
This last term is bounded above by 
\[ \E_{\gamma \in \Ghat} \abs{\widehat{\mu_A}(\gamma)} \abs{\widehat{1_L}(\gamma)} \abs{\widehat{\mu_{T-T}}(\gamma)}^k \abs{\gamma(t)-1}, \]
and so if $t \in V := \Spec_\eta(\mu_{T-T})^\perp := \{ \gamma \in \Ghat : \abs{\widehat{\mu_{T-T}}(\gamma)} \geq \eta\}^\perp$ then this is at most 
\[ 2\eta^k \abs{A}^{-1/2}\abs{L}^{1/2} \leq 2\eta^k K^{1/2}. \]
If we choose $\eta=1/2$, say, and $k\approx C \log(2K/\epsilon)$, then this implies that for $t\in V$, 
\[\norm{\mu_A\ast 1_L(\cdot +t)-\mu_A\ast 1_L}_p
\leq 2k \epsilon \norm{\mu_A\ast 1_L}_{p/2}^{1/2} + 4k\epsilon^2.\]
The proof is complete since $\dim \Spec_{1/2}(\mu_{T-T})\ll \log(1/\mu(T))$ by Chang's theorem.
\end{proof}

\section{Central moments of the binomial distribution}\label{app:moments}

Here we prove Lemma \ref{lemma:probsampMeasure}, a version of the sampling lemma at the heart of the probabilistic approach to almost-periodicity. As mentioned before, it is a variant of results from \cite{CrSi:2010}.

\begin{lemma}\label{lemma:probsampMeasureApp}
Let $m, k \geq 1$. Let $A, L$ be finite measure subsets of a $\sigma$-finite locally compact group $G$, let $\mu$ be a $\sigma$-finite Borel measure on $G$, and denote 
\[ f = \mu_A*1_L \cdot (1 - \mu_A*1_L). \]
If $\tup{a} \in A^k$ is sampled uniformly at random, then, provided $k \geq Cm/\epsilon^2$,
\[ \E \norm{ \mu_\tup{a}*1_L - \mu_A*1_L }_{L^{2m}(\mu)}^{2m} \leq \epsilon^{2m} \norm{f}_{L^m(\mu)}^m + \epsilon^{4m-2} \norm{f}_{L^1(\mu)}. \]
\end{lemma}
Note that the measures of $A$ and $L$, the $\sigma$-finiteness, and the convolutions are with respect to (left) Haar measure $\mu_G$ on $G$. Thus
\[ f*g(x) = \int f(y) g(y^{-1}x) d\mu_G(y). \]
The function $\mu_\tup{a}*1_L$ is to be interpreted as
\[ \mu_\tup{a}*1_L(x) = \E_{j \in [k]} 1_L(a_j^{-1}x). \]
We remark that although introducing the function $f$ might seem cumbersome, it turns out to be somewhat natural. Note for example that if $A = L$ is a subgroup, the right-hand side is actually $0$, since then $\mu_A*1_A = 1_A$.

To prove this lemma, we shall use the following bounds for the central moments of the binomial distribution. These are surely standard, but we include a self-contained proof as we have not been able to locate a readily available reference. (We note that they follow from general results on iid random variables, but only after some calculation.) 

\begin{lemma}\label{lemma:binmoments}
Let $p \in [0,1]$ and $m, n \in \N$. If $X$ is a $\Bin(n, p)$ random variable, with $q = 1-p$, then
\[ \E \abs{X - np}^{2m} \leq m \max(m^{2m-1} npq, e^{m-1}(m npq)^m). \]
In particular, if $Z = X/n$ and $n \geq 4m/\delta$, we have 
\[ \E \abs{Z - p}^{2m} \leq \delta^m (pq)^m + \delta^{2m-1} pq. \]
\end{lemma}

The particular constants here could be improved, but are of no consequence to us. Before proving this, let us see how it implies Lemma \ref{lemma:probsampMeasureApp}.
\begin{proof}[Proof of Lemma \ref{lemma:probsampMeasureApp}]
Fix $x \in G$. For $\tup{a} = (a_1, \ldots, a_k)$ sampled uniformly from $A^k$, we have
\[ \mu_\tup{a}*1_L(x) = \E_{j \in [k]} 1_L(a_j^{-1} x). \]
This is an average of $k$ Bernoulli random variables $1_L(a_j^{-1} x)$, each with parameter
\[ p = \E 1_L(a_j^{-1} x) = \mu_A*1_L(x). \]
The sum of these $k$ Bernoulli random variables is a binomial random variable, and so Lemma \ref{lemma:binmoments} (with $n = k$) implies that
\[ \E \abs{ \mu_\tup{a}*1_L(x) - \mu_A*1_L(x) }^{2m} \leq \epsilon^{2m} f(x)^m + \epsilon^{2m-1} f(x). \]
Integrating over all $x \in G$ with respect to $\mu$ and swapping orders of integration using Fubini--Tonelli yields the result.
\end{proof}

To prove the above moment bounds, we use a few standard facts about a binomially distributed random variable $X \sim \Bin(n,p)$. Throughout, let 
\[ \mu_r = \E (X - np)^r = \sum_{j=0}^n \binom{n}{j} p^j q^{n-j} (j-np)^r. \]
The moment generating function of $X-np$ is 
\[ \sum_{k=0}^{\infty} \frac{\mu_k t^k}{k!} = \left( q e^{-t p} + p e^{t q} \right)^n.
\]
We note that $\mu_r \geq 0$ provided $p \leq 1/2$. Furthermore, formal manipulation of the above power series yields, as noted in \cite[\S5.5]{kendall-stuart}, the recurrence
\begin{equation}
\mu_r = npq\sum_{j=0}^{r-2} \binom{r-1}{j} \mu_j - p \sum_{j=0}^{r-2} \binom{r-1}{j} \mu_{j+1} \label{eqn:mu_rec}
\end{equation}
for $r \geq 2$, which, together with the initial conditions $\mu_0 = 1$, $\mu_1 = 0$ can be used to compute these moments. We use it to bound the moments as follows.

\begin{proposition}[Polynomial bound for central moments]\label{prop:polymoments}
For $p \leq 1/2$, the $r$-th central moment of a $\Bin(n,p)$ random variable satisfies 
\[ \mu_r \leq \nu_r(npq), \]
where $\nu_r(x)$ is a polynomial defined recursively by $\nu_0 = 1$, $\nu_1 = 0$ and
\[ \nu_r = x \sum_{j=0}^{r-2} \binom{r-1}{j} \nu_j. \]
\end{proposition}
\begin{proof}
For $p \leq 1/2$, all the moments are non-negative, and so \eqref{eqn:mu_rec} yields
\[ \mu_r \leq npq\sum_{j=0}^{r-2} \binom{r-1}{j} \mu_j. \]
The claim thus follows by induction.
\end{proof}

The polynomials $\nu_r$ so defined give the best upper bound possible for $\mu_r$ that is a polynomial in $npq$ and otherwise uniform in $p$. 
We can describe them fairly explicitly:

\begin{proposition}[Explicit description of the polynomials $\nu_r$]\label{prop:nuformula}
For $r \geq 0$, 
\[ \nu_r = \sum_{k \geq 0} \stirlingtwo{r}{k} x^k \]
where $\stirlingtwo{r}{k}$ is a $2$-associated Stirling number of the second kind, defined as the number of partitions of a set of size $r$ into $k$ parts, each of size at least $2$. In particular, $\nu_r$ has degree $\floor{r/2}$ and, if $r \geq 1$, no constant term.
\end{proposition}

For clarity surrounding edge cases, we take $\stirlingtwo{0}{0} = 1$ and $\stirlingtwo{r}{0} = 0 = \stirlingtwo{0}{k}$ for $r, k \geq 1$. To prove the proposition, we note the following recurrence for $\stirlingtwo{r}{k}$.

\begin{lemma}\label{lemma:stir2rec}
For $r \geq 0$ and $k \geq 1$,
\[ \stirlingtwo{r}{k} = \sum_{j=0}^{r-2} \binom{r-1}{j} \stirlingtwo{j}{k-1}. \]
\end{lemma}
\begin{proof}
For $r \leq 1$ the result is trivial, so assume $r \geq 2$. We consider the partitions of $[r]$ into $k$ parts, each of size $2$. We count these according to how many elements $1$ is placed with. If the part containing $1$ is to have size $n+1$, there are $\binom{r-1}{n}$ choices for the other elements to place with $1$, and $\stirlingtwo{r-1-n}{k-1}$ ways to partition the remaining elements into $k-1$ parts, each of size at least $2$. Summing up all these (disjoint) ways yields the result. 
\end{proof}

\begin{proof}[Proof of Proposition \ref{prop:nuformula}]
The recursion in Lemma \ref{lemma:stir2rec} shows immediately that the sequence $p_r = \sum_{k \geq 0} \stirlingtwo{r}{k} x^k$ satisfies the recursion defining $\nu_r$. Since the initial conditions also match, the sequences are the same.
\end{proof}

We next use this combinatorial description to place an upper bound on $\nu_r$.

\begin{proposition}[Upper bound for $\nu_r$]\label{prop:nubound}
For $x \geq 0$,
\[ \nu_{2m} \leq m \max\left( e^{m-1} (mx)^m, m^{2m-1} x\right). \]
\end{proposition}
\begin{proof}
By Proposition \ref{prop:nuformula},
\[ \nu_{2m} = \sum_{k = 1}^m \stirlingtwo{2m}{k} x^k. \]
Using the crude bounds 
\[ \stirlingtwo{2m}{k} \leq k^{2m}/k! \leq e^{k-1} k^{2m-k} \leq e^{-1} m^{2m} (e/m)^k, \]
valid for $1 \leq k \leq m$, we have
\[ \nu_{2m} \leq e^{-1} m^{2m} \sum_{k = 1}^m (x e/m)^k \leq e^{-1} m^{2m+1} \max( xe/m, (xe/m)^m). \]
Rearranging, this completes the proof.
\end{proof}

One could of course be more careful here in order to obtain better constants, but we have no need for it, opting instead for uniform bounds.

\begin{proof}[Proof of Lemma \ref{lemma:binmoments}]
The first claim follows immediately from combining Proposition \ref{prop:polymoments} and Proposition \ref{prop:nubound}. The second one follows from the first upon replacing the maximum by a sum.
\end{proof}

\section*{Acknowledgements}
The first-named author was supported by both the Heilbronn Institute for Mathematical Research, Bristol, UK, and a postdoctoral grant funded by the Royal Society while this work was completed. The second-named author was supported by The Swedish Research Council grant 2013-4896. The authors would like to thank the Harvard CMSA for its hospitality during its Combinatorics and Complexity programme, where part of this work was carried out.

\bibliographystyle{amsplain}


\begin{dajauthors}
\begin{authorinfo}[tfb]
  Thomas F. Bloom\\
  Department of Pure Mathematics and Mathematical Statistics\\
  University of Cambridge\\
  Cambridge, UK\\
  \href{mailto:tb634@cam.ac.uk}{\texttt{tb634@cam.ac.uk}}

\end{authorinfo}
\begin{authorinfo}[os]
  Olof Sisask\\
  Department of Mathematics\\
  Uppsala University\\
  Sweden\\
  \href{mailto:olof.sisask@math.uu.se}{\texttt{olof.sisask@math.uu.se}}
\end{authorinfo}
\end{dajauthors}

\end{document}